\documentclass[reqno,11pt]{amsart}
 
\usepackage{amsmath}
\allowdisplaybreaks[1]
\usepackage{amsthm}  
\usepackage{verbatim}
\usepackage{enumerate} 
\usepackage{mathtools}  
\usepackage{amssymb}
\usepackage{mathrsfs}
\usepackage[top=1.5in, bottom=1.5in, left=0.7in, right=0.7in]{geometry}  
\usepackage[colorlinks]{hyperref}
\hypersetup{
	colorlinks,
	citecolor=blue,
	linkcolor=red
}   
\numberwithin{equation}{section}
\usepackage{xypic}
\usepackage{bm}
\usepackage{tikz}
\usepackage{float}

\newtheorem{theorem}{\textbf{Theorem}}[section]
\newtheorem{theorem*}{\textbf{Theorem}}

\newtheorem{definition}[theorem]{\textbf{Definition}}
\newtheorem{proposition}[theorem]{\textbf{Proposition}}

\newtheorem{question}[theorem]{\textbf{Question}}

\newtheorem{claim}[theorem*]{\textbf{Claim}}

\newtheorem{corollary}[theorem]{\textbf{Corollary}}
\newtheorem{remark}[theorem]{\textbf{Remark}}

\newtheorem{example}[theorem]{\textbf{Example}}
\newtheorem{conjecture}[theorem]{\textbf{Conjecture}}
\newtheorem{definition/proposition}[theorem]{\textbf{Definition/Proposition}}
\newtheorem{innercustomgeneric}{\customgenericname}
\providecommand{\customgenericname}{}
\newcommand{\newcustomtheorem}[2]{%
	\newenvironment{#1}[1]
	{%
		\renewcommand\customgenericname{#2}%
		\renewcommand\theinnercustomgeneric{##1}%
		\innercustomgeneric
	}
	{\endinnercustomgeneric}
}
\newcustomtheorem{customthm}{Theorem}
\newcustomtheorem{customprop}{Proposition}
\newcustomtheorem{customcoro}{Corollary}

\def\N{{\mathbb N}}
\def\R{\mathbb{R}}
\def\Z{{\mathbb Z}}

\def\C{{\mathbb C}}
\def\D{{\mathbb D}}

\newcommand{\CP}{\mathbb{C}\mathbb{P}}

\def\cA{{\mathcal A}}

\def\cL{{\mathcal L}}
\def\cM{{\mathcal M}}

\def\cO{{\mathcal O}}

\def\rd{{\rm d}}

\def\la{\langle\,}
\def\ra{\,\rangle}

\DeclareMathOperator{\Ima}{im}

\DeclareMathOperator{\Hom}{Hom}

\DeclareMathOperator{\vol}{vol}

\DeclareMathOperator{\Crit}{Crit}
\title{On the minimal symplectic area of Lagrangians}
\author{Zhengyi Zhou}

\begin{document}
	\maketitle
\begin{abstract}
We show that the minimal symplectic area of Lagrangian submanifolds are universally bounded in symplectically aspherical domains with vanishing symplectic cohomology. If an exact domain admits a $k$-semi-dilation, then the minimal symplectic area is universally bounded for $K(\pi,1)$-Lagrangians. As a corollary, we show that the Arnol'd chord conjecture holds for the following four cases: (1) $Y$ admits an exact filling with $SH^*(W)=0$ (for some nonzero ring coefficient); (2) $Y$ admits a symplectically aspherical filling with $SH^*(W)=0$ and simply connected Legendrians; (3) $Y$ admits an exact filling with a $k$-semi-dilation and the Legendrian is a $K(\pi,1)$ space; (4) $Y$ is the cosphere bundle $S^*Q$ with $\pi_2(Q)\to H_2(Q)$ nontrivial and the Legendrian has trivial $\pi_2$.  In addition, we obtain the existence of homoclinic orbits in case (1). We also provide many more examples with $k$-semi-dilations in all dimensions $\ge 4$.
\end{abstract}
\section{Introduction}
One of the fundamental questions in contact geometry is the following chord conjecture by Arnol'd \cite{arnol1986first}.
\begin{conjecture}
	Let $(Y,\xi)$ be a closed contact manifold, then any closed Legendrian carries a Reeb chord for any contact form $\alpha$.
\end{conjecture}
The conjecture was completely addressed in dimension $3$ by Hutchings-Taubes \cite{hutchings2013proof}. The surgical picture in \cite{hutchings2013proof} first appeared in Cieliebak's proof \cite{cieliebak2002handle} of the chord conjecture for some special cases, as the Reeb chord of a Legendrian sphere is closely related to the surgery formula for Floer theories \cite{bourgeois2012effect}. However, in dimension $\ge 5$, a Legendrian can have a diffeomorphism type different from spheres, hence the surgery point of view may not apply in general. On the other hand, Mohnke used a different method and proved the chord conjecture for any subcritically fillable contact manifolds in all dimensions \cite{mohnke2001holomorphic}.  

Mohnke's argument revealed a closely related concept called minimal symplectic area \cite{cieliebak2018punctured}. The minimal symplectic area of a closed Lagrangian $L$ inside a symplectic manifold $(W,\omega)$ is defined to be
\begin{equation}\label{eqn:Amin}
A_{\min}(L,W):=\inf\left\{\int u^*\omega\left|u\in \pi_2(W,L), \int u^*\omega > 0\right.\right\}\in [0,\infty].
\end{equation}
Let $\cL$ be a set of diffeomorphism types of closed connected manifolds, we define 
$$A_{\min}(\cL,W):=\sup\left\{ A_{\min}(L,W)\left|L\text{ is a Lagrangian in } W, \text{the diffeomorphism type of } L \text{ is in }\cL\right.\right\}.$$
There are two collections that we are interested in, $\cL_{all}$ the collection of all $n$-manifolds, $\cL_{K(\pi,1)}$ the collection of all $n$-dimensional $K(\pi,1)$-spaces, in particular, those admitting a non-positive sectional curvature. The following theorem is due to Mohnke \cite{mohnke2001holomorphic}.
\begin{theorem}[\cite{mohnke2001holomorphic}]\label{thm:chord}
	If the contact manifold $Y$ admits an \emph{exact} filling $W$, such that $A_{\min}(\cL_{all},W)<\infty$, then the chord conjecture holds for $Y$. If $A_{\min}(\cL_{K(\pi,1)},W)<\infty$, then the chord conjecture holds for $K(\pi,1)$ Legendrians\footnote{Such statement is not the minimal requirement for the argument in \cite{mohnke2001holomorphic} to work.}.
\end{theorem}	
Then the chord conjecture for subcritically fillable contact manifolds follows from that the minimal symplectic area is bounded from above by the displacement energy of the subcritical domain $W$ by \cite{chekanov1998lagrangian}.  Our main theorem is finding new families of symplectic manifolds with universally bounded minimal symplectic area, which are not necessarily displaceable.

\begin{theorem}\label{thm:main}
	Let $W$ be an symplectically aspherical domain and $R$ be unital commutative ring such that $1\ne 0 \in R$.
	\begin{enumerate}
		\item\label{c1} If $SH^*(W;R)=0$, then $A_{\min}(\cL_{all},W)<\infty$.
		\item\label{c2} If $W$ is an exact domain and admits a $k$-semi-dilation in $R$ coefficient \cite{zhou2019symplectic}, then $A_{\min}(\cL_{K(\pi,1)},W)<\infty$.
	\end{enumerate}
\end{theorem}

\begin{remark}
	It is not true that $W'\subset W$ implies that $A_{\min}(\cL,W')\le A_{\min}(\cL,W)$ since $\pi_2(W',L)\to \pi_2(W,L)$ may not be surjective. For example, if $\cL=\{T^n\}$, then $A_{\min}(\cL,B^{2n}(1))<\infty$. On the other hand $A_{\min}(\cL,D_{\epsilon}^*T^n)=\infty$ as it contains an exact $T^n$, but the disk bundle $D_{\epsilon}^*T^n$ with small enough radius $\epsilon$ can be symplectically embedded into the ball $B^{2n}(1)$. However if $W'\subset W$ is a homotopy equivalence, then we have $A_{\min}(\cL,W')\le A_{\min}(\cL,W)$. In particular, we can use $A_{\min}(\cL,W)$ to define a generalized symplectic capacity for star shaped domains in $\C^n$, e.g.\ the Lagrangian capacity \cite{hofer2007quantitative,cieliebak2018punctured}.
\end{remark}

By \cite{kang2014symplectic}, if an exact domain $W$ is displaceable, then $SH^*(W;\Z)=0$. But the vanishing of symplectic cohomology does not imply being displaceable. Although the vanishing of symplectic cohomology is a restriction, there are still many examples (1) flexible Weinstein domains \cite{bourgeois2012effect,murphy2018subflexible}, (2) subflexible Weinstein domains \cite{murphy2018subflexible}, (3) $V\times \D$ for any exact domain $V$ \cite{oancea2006kunneth}, (4) prequantization line bundles over symplectically aspherical manifolds \cite{oancea2008fibered}. Moreover, there are many non-flexible examples whose symplectic cohomology vanishes in certain finite field coefficient \cite{abouzaid2010altering,lazarev2020prime}.  The $k$-semi-dilation is a generalization of the vanishing of symplectic cohomology, as admitting a $0$-semi-dilation is equivalent to the vanishing of symplectic cohomology. Examples with $k$-semi-dilations include cotangent bundles of simply connected manifolds \cite[Proposition 5.1]{zhou2019symplectic}, more generally tree plumbings of cotangent bundles of simply connected (spin) manifolds of dimension at least $3$ \cite[Proposition 6.2]{li2019exact}, and Milnor fibers of many canonical singularities \cite[Theorem A]{zhou2019symplectic}. Moreover, the $k$-semi-dilation is preserved under subcritical and flexible surgeries under mild conditions, products and Lefschetz fibrations \cite[\S 3.6]{zhou2019symplectic}. In this paper, we also find more examples with $k$-semi-dilations, see Theorem \ref{thm:SD}. Therefore there is a rich class of examples that Theorem \ref{thm:main} can be applied to. There are also examples of vanishing of symplectic cohomology using local systems \cite{albers2017local}. As a corollary we have the following.
\begin{theorem}\label{thm:local}
 Let $Q$ be a closed manifold such that the Hurewicz map $\pi_2(Q)\to H_2(Q)$ is nontrivial. Then $A_{\min}(L,D^*Q)$ is uniformly bounded for any Lagrangian $L$ such that $\pi_2(L)\to \pi_2(T^*Q)$ is trivial.
\end{theorem}
The topological conditions in \eqref{c2} of Theorem \ref{thm:main} and \ref{thm:local} are necessary. For example, the zero section of $D^*S^2$ is exact, hence has infinite minimal symplectic area. However, \eqref{c2} of Theorem \ref{thm:main} or \ref{thm:local} imply that oriented Lagrangians that are not spheres have a universal minimal symplectic area bound. Indeed, there are many such Lagrangians, as we can take several copies of the zero section after different Hamiltonian perturbations and then use Lagrangian surgeries to resolve the intersections to get smooth Lagrangians with higher genus.

Combining Theorem \ref{thm:chord}, \ref{thm:main}, and \ref{thm:local} together, we prove the chord conjecture for the following examples.
\begin{corollary}\label{cor:chord}
	Let $Y$ be a contact manifold and $R\ne 0$ a commutative ring.
	\begin{enumerate}
		\item\label{ch1} If $Y$ has an exact filling $W$ with $SH^*(W;R)=0$,  then the chord conjecture holds for $Y$. 
		\item\label{ch2} If $Y$ has a symplectically aspherical filling $W$ with $SH^*(W;R)=0$,  then the chord conjecture holds for any Legendrian $\Lambda$ such that $\pi_1(\Lambda)\to \pi_1(W)$ is trivial, in particular, any simply connected Legendrian.
		\item\label{ch3} If $Y$ has an exact filling $W$ with a $k$-semi-dilation in $R$-coefficient, then the chord conjecture holds for any $K(\pi,1)$ Legendrian.
		\item\label{ch4} If $Y=S^*Q$ with  the Hurewicz map $\pi_2(Q)\to H_2(Q)$ nontrivial, then the chord conjecture holds for any Legendrian $\Lambda$ such that $\pi_2(\Lambda)\to \pi_2(T^*Q)$ is trivial.
	\end{enumerate}
\end{corollary}

\begin{remark}
	For exact domains $W$ with $SH^*(W)=0$, Ritter showed that the chord conjecture holds for exactly fillable Legendrians via showing the vanishing of the wrapped Floer cohomology of the filling \cite{ritter2013topological}. However, not every Legendrian is fillable and we do not have an effective criterion to prove a Legendrian is fillable. 
\end{remark}

We first explain the mechanism behind the proof of Theorem \ref{thm:main}. In the extreme case, if $L$ is an exact Lagrangian inside an exact domain $W$, in particular, then $A_{\min}(L,W)=\infty$. In this case, we have the Viterbo transfer map $SH^*(W)\to SH^*(T^*L)$ preserving various structures. Note that $SH^*(T^*L)$ is never zero and $SH^*(T^*L)$ does not carry a $k$-semi-dilation if $L$ is a $K(\pi,1)$ space. Using the Viterbo transfer map, if $SH^*(W)=0$ then $W$ has no exact Lagrangians, and if $W$ admits a $k$-semi-dilation, then $W$ has no exact $K(\pi,1)$ Lagrangians. If $L$ is not an exact Lagrangian, the Viterbo transfer map (without deformation) does not exist. However a truncated version of the Viterbo transfer map still exists, where the threshold of the truncation depends on $A_{\min}(L,W)$. Then Theorem \ref{thm:main} follows from the same argument as in the exact case above.

Inspired by the argument in \cite{mohnke2001holomorphic}, it is natural to look for the extreme case of the failure of finite minimal symplectic area in the symplectization, i.e.\ exact Lagrangians in a symplectization. By \cite{murphy2013closed}, there exist abundant exact Lagrangians in overtwisted contact manifolds of $\dim \ge 3$. Therefore the approach in \cite{mohnke2001holomorphic} is not applicable to all contact manifolds. On the other hand, every exact Lagrangian in the symplectization is displaceable by Hamiltonian diffeomorphisms \cite[Remark 4]{murphy2013closed}. In particular, there is no exact Lagrangian in the symplectization if the contact manifold is exactly fillable, for otherwise the Lagrangian Floer cohomology is well-defined and non-vanishing, contradicting that it is displaceable. It is an interesting question to understand if there is some quantitative shadow of this argument which leads to existence of Reeb chords.

Mohnke's construction was modified by Lisi \cite{MR2418284} to obtain existence of homoclinic orbits in the case of displaceable exact domains. Here the displaceable property is again used to obtain universal minimal symplectic area upper bounds by \cite{chekanov1998lagrangian}. With new minimal symplectic area bound given by Theorem \ref{thm:main}, we obtain the following result.

\begin{corollary}\label{cor:homo}
	Assume $(W,\lambda)$ is an exact domain such that $SH^*(W;R)=0$ for a commutative ring $R\ne 0$. Let $H:W\to \R$ be a Hamiltonian with $H(x_0)=0$ and $x_0$ a hyperbolic zero of the Hamiltonian vector field $X_H$. Suppose $H^{-1}(0)$ is compact and $H^{-1}(0)\backslash\{x_0\}$ is of restricted contact type, i.e.\ $\lambda(X_H)>0$ on $H^{-1}(0)\backslash\{x_0\}$, then there is an orbit of $X_H$ homoclinic to $x_0$.
\end{corollary}

We also find many new examples with $k$-semi-dilations, which provide more examples for Theorem \ref{thm:main} and Corollary \ref{cor:chord}.
\begin{theorem}\label{thm:SD}
	Let $X$ be a degree $m$ smooth hypersurface in $\CP^{n+1}$ for $m\le n$. Then for any holomorphic section $s$ of $\cO(k)$ for $k\le n+1-m$, the affine variety $X\backslash s^{-1}(0)$ admits a $k+m-2$-semi-dilation.
\end{theorem}
\subsection*{Acknowledgements}
The author would like to thank Laurent C\^ot\'e for explaining the non-existence of exact Lagrangians in the symplectization of an exact fillable contact manifold and pointing out \cite{MR2418284, murphy2013closed}. The author is also in debt to the referee for pointing out a mistake in an earlier draft. The author is supported by the National Science Foundation under Grant No. DMS-1926686. It is a great pleasure to acknowledge the Institute for Advanced Study for its warm hospitality.

\section{Symplectic cohomology}
In this section, we review briefly the construction of symplectic cohomology for exact and symplectically aspherical domains and the Viterbo transfer map following \cite[\S 3]{cieliebak2018symplectic}. We will only recall basics of symplectic cohomology to set up notations and relevant structures for our main results. We refer readers to \cite{cieliebak2018symplectic,ritter2013topological,seidel2006biased} for a more complete treatment of the subject.

\subsection{Symplectic cohomology}
\subsubsection{Symplectic cohomology for exact domains}
Let $(W,\lambda)$ be an exact filling and $(\widehat{W},\widehat{\lambda})=W\cup \partial W \times (1,\infty)$ be the completion. Let $H$ be a time-dependent Hamiltonian on the completion $(\widehat{W},\widehat{\lambda})$, then the symplectic action for an orbit $\gamma$ is 
\begin{equation}\label{eqn:action}
\cA_H(\gamma)=-\int \gamma^*\widehat{\lambda}+\int_{S^1} (H\circ \gamma)\rd t,
\end{equation}
where $\rd \widehat{\lambda}(\cdot, X_{H})=\rd H$. We say an almost complex structure $J$ is cylindrically convex near $\partial W \times \{r_0\}$ iff near the hypersurface $r=r_0$ we have that $\widehat{\lambda}\circ J =\rd r$. We will consider a Hamiltonian $H$, which is a $C^2$ small perturbation (for the more precise meaning, see \eqref{ii} below) to the Hamiltonian that is $0$ on $W$ and linear with slope $a$ on $\partial W \times (1,\infty)$, such that $a$ is not a period of Reeb orbits of $R_{\lambda}$ on $\partial W\times \{1\}$. We may assume the Hamiltonian is non-degenerate, then the periodic orbits consist of the following.
\begin{enumerate}[(i)]
	\item\label{i} Constant orbits on $W$ with $\cA_H\approx 0$.
	\item\label{ii} Non-constant orbits near Reeb orbits of $R_{\lambda}$ on $\partial W \times \{1\}$, with action close to the negative period of the Reeb orbits. In particular, we have $\cA_H\in (-a,0)$.  To see this, note that our $H$ is an $S^1$-dependent $C^2$ small perturbation to an autonomous Hamiltonian $h(r)$ with $h'(r)=a$ for $r>1+\epsilon$ with $\epsilon$ small and $h(r)=0$ for $r\le 1$. The non-constant orbits of $h(r)$ are in $S^1$ families like $(\xi(h'(r_0)t),r_0)$, where $\xi$ is a Reeb orbit of $(\partial W,\lambda|_{\partial W})$ with the period of $\xi$ is $h'(r_0)$ for $1<r_0<1+\epsilon$. Then the $S^1$ family comes from the reparameterization of the Reeb orbit $\xi$. Therefore the symplectic action of such orbit is
	$$-h'(r_0)r_0+h(r_0).$$
	It is clear when $\epsilon\ll 1$, we have that the symplectic action is approximately the negative period of $\xi$, which, in particular, is in $(-a,0)$. Then the $C^2$-small perturbation to $h(r)$ in \cite[Lemma 3.3]{MR2475400} will break the $S^1$ family orbits into two non-degenerate orbits with symplectic action arbitrarily close to the original $S^1$ family.
\end{enumerate} 
After fixing an $S^1$-dependent compatible almost complex structure $J$ that is cylindrically convex near a slice (i.e.\ a hypersurface $r=r_0$) where the Hamiltonian is linear with slope $a$, we can consider the compactified moduli space of Floer cylinders, i.e.\ solutions to $\partial_s u+J(\partial_t u-X_{H})=0$ modulo the $\R$ translation and asymptotic to two Hamiltonian orbits
$$\cM_{x,y}=\overline{\left\{u:\R_s\times S^1_t\to \widehat{W}\left|\partial_s u+J(\partial_t u-X_{H})=0, \lim_{s\to \infty} u(s,\cdot)=x, \lim_{s\to -\infty} u(s,\cdot)=y \right.\right\}/\R}.$$
With a generic choice of $J$, the count of rigid Floer cylinders defines a cochain complex $C(H)$, which is a free $R$ module generated by Hamiltonian orbits,  for any commutative ring $R$ (our default setting is $R=\Z$). The differential $\delta^0$ is defined as
$$\delta^0(x)=\sum_{y,\dim \cM_{x,y}=0} \left(\# \cM_{x,y}\right) y.$$
The orbits of type \eqref{i} form a subcomplex $C_0(H)$, whose cohomology is $H^*(W)$. The orbits of type \eqref{ii} form a quotient complex $C_+(H)$. The cochain complexes are graded by $n-\mu_{CZ}$, which is in general a $\Z/2\Z$ grading unless $c_1(W)=0$. Given two Hamiltonians $H_a,H_b$ with slopes $a<b$, we can consider a non-increasing homotopy of Hamiltonians $H_s$ from $H_b$ to $H_a$, i.e.\ $H_s=H_b$ for $s\ll0$ and $H_s=H_a$ for $s\gg 0$. Then the count of rigid solutions to the parameterized Floer's equation $\partial_s u+J(\partial_t u-X_{H_s})=0$ defines a continuation map $C(H_a)\to C(H_b)$, which is compatible with splitting into zero and positive complexes. Then the (positive) symplectic cohomology of $W$ is defined as
$$SH^*(W):=\lim_{a\to \infty} H^*(C(H_a)),\quad SH_+^*(W):=\lim_{a\to \infty} H^*(C_+(H_a)),$$
which fit into a tautological exact sequence,
$$\ldots \to H^*(W)\to SH^*(W)\to SH^*_+(W)\to H^{*+1}(W)\to \ldots $$
We define the filtered symplectic cohomology $SH^*_{<a}(W)$ by $H^*(C(H_a))$ and $SH^*_{+,<a}(W)$ by $H^*(C_+(H_a))$, which are independent of $H_a$ as long as $a$ is not a period of Reeb orbits on $\partial W$, see \cite[Proposition 2.8]{zhou2020mathbb}. We use $SH^{[0],*}(W),SH^{[0],*}_+(W)$ to denote the cohomology generated by \emph{contractible} orbits.

\subsubsection{Symplectic cohomology for symplectically aspherical domains} Assume $(W,\omega)$ is a strong domain\footnote{A strong domain $(W,\omega)$ is a compact symplectic manifold with boundary, such that $\omega=\rd\lambda$ near $\partial W$ and the Liouville vector field $X$ defined by $\iota_X \omega =\lambda$ near $\partial W$ points out along $\partial W$.} that is symplectically aspherical, i.e.\ $\omega|_{\pi_2(W)}=0$. Then the symplectic action for a Hamiltonian $H$ and a \emph{contractible} orbit $\gamma$ is now 
\begin{equation}\label{eqn:action'}
\cA_H(\gamma)=-\int_D u^*\widehat{\omega}+\int_{S^1} (H\circ \gamma) \rd t,
\end{equation}
where $u:D\to \widehat{W}$ is an extension of $\gamma$. \eqref{eqn:action'} is independent of $u$ since $\omega|_{\pi_2(W)}=0$. In this case, we still have a cochain complex $C^{[0]}(H_a)$ for a Hamiltonian with slope $a$ and a continuation map $C^{[0]}(H_a)\to C^{[0]}(H_b)$ for $a<b$. However, it is not longer clear from the symplectic action perspective that $C^{[0]}(H_a)$ splits into $C_0$ and $C_+$\footnote{The splitting, in particular, the positive symplectic cohomology is still defined by the asymptotic behaviour lemma \cite[Lemma 2.3]{cieliebak2018symplectic}.}. Nevertheless, we can still define symplectic cohomology $SH^{[0],*}(W):=\lim_{a\to \infty} H^{*}(C^{[0]}(H_a))$ and filtered symplectic cohomology $SH^{[0],*}_{< a}(W):= H^*(C^{[0]}(H_a))$ and a continuation map $H^{*}(W)=SH^{[0],*}_{<\epsilon}(W)\to SH^{[0],*}_{<a}(W)\to SH^{[0],*}(W)$ for $\epsilon\ll 1$, which in the exact case, is the map $H^*(W)\to SH^{[0],*}_{<a}(W)\to SH^{[0],*}(W)$ in the tautological exact sequence.

\subsection{The Viterbo transfer map \cite[\S 5]{cieliebak2018symplectic}}\label{ss22}
Let $(V,\lambda_V)\subset (W,\lambda_W)$ be an exact subdomain, i.e.\ $\lambda_W|_V=\lambda_V$, then there exists $\epsilon>0$, such that $V_{\epsilon}:=V\cup \partial V \times (1,1+\epsilon]\subset W$ with $\lambda_W|_{V_{\epsilon}}=\widehat{\lambda}_V$. Then we can consider a Hamiltonian $H$ on $\widehat{W}$ as a $C^2$-small non-degenerate perturbation to the following.
\begin{enumerate}
	\item $H$ is $0$ on $V$.
	\item $H$ is linear with slope $B$ on $\partial V \times [1,1+\epsilon]$, such that $B$ is not the period of a Reeb orbit on $\partial V$.
	\item $H$ is $B\epsilon$ on $W\backslash V_{\epsilon}$.
	\item $H$ is linear with slope $A\le B\epsilon$ on $\partial W \times [1,\infty)$, such that $A$ is not the period of a Reeb orbit on $\partial W$.  
\end{enumerate}
Then there are five classes of periodic orbits of $X_H$.
\begin{enumerate}[(I)]
	\item\label{I} Constant orbits on $V$ with $\cA_H\approx 0$.
	\item\label{II} Non-constant periodic orbits near $\partial V$ with $\cA_H \in (-B,0)$.
	\item\label{III} Non-constant periodic orbits near $\partial V \times \{1+\epsilon\}$ with action $\cA_H\in (B\epsilon-(1+\epsilon)B, B\epsilon)=(-B,B\epsilon)$. To see the action region, it is similar to \eqref{ii} after \eqref{eqn:action}. Those orbits are close to Reeb orbits of $(\partial W, \lambda|_{\partial W})$, but placed near $r=1+\epsilon$. Therefore $-\int \gamma^*\widehat{\lambda}$ in \eqref{eqn:action} is close to $(1+\epsilon)$  times (as $\widehat{\lambda}|_{r=1+\epsilon}=(1+\epsilon)\lambda|_{\partial W}$) the period of the Reeb orbits and $\int_{S^1} (H\circ \gamma) \rd t$ is close to $B\epsilon$. Hence the claim follows.
	\item\label{IV} Constant orbits on $W\backslash V_{\epsilon}$ with $\cA_H\approx B\epsilon$.
	\item\label{V} Non-constant periodic orbits near $\partial W$ with $\cA_H\in (B\epsilon-A, B\epsilon)$
\end{enumerate}
In particular, when $B\epsilon \ge A$, the quotient complex generated by orbits with non-positive action are generated by type \eqref{I}, \eqref{II} orbits along with some of the type \eqref{III} orbits. However, there is no Floer cylinder from a type \eqref{III} orbit to a type \eqref{I} or \eqref{II} orbit \cite[Figure 6]{cieliebak2018symplectic}. This follows from the asymptotic behavior lemma \cite[Lemma 2.3]{cieliebak2018symplectic} and the integrated maximum principle \cite[Lemma 2.2]{cieliebak2018symplectic}. Therefore orbits of the form \eqref{I}, \eqref{II} form a quotient complex when $B\epsilon\ge A$. Let $H_V$ denote the Hamiltonian on $\widehat{V}$ which is the linear extension of the truncation of $H$ on $V_{\epsilon}$. Then by \cite[Lemma 2.2]{cieliebak2018symplectic}, the quotient complex is identified with $C^*(H_V)$. Those two applications of the integrated maximum principle are where exactness of the cobordism $W\backslash V$ is crucial. Next we consider a Hamiltonian $H_W$ on $\widehat{W}$ which is a $C^2$ small perturbation to the function that is zero on $W$ and linear with slope $A$ on $\partial W\times (1,\infty)$. Then $H_W\le H$ and we can find non-increasing homotopy from $H$ to $H_W$, which defines a continuation map. Therefore we have a map 
$$C^*(H_W)\to C^*(H) \to C^*(H_V).$$
Since the continuation map increases the symplectic action, the above map respects the splitting into $C_0,C_+$. Taking direct limit for $B$ yields the Viterbo transfer map which is compatible with the tautological exact sequence,
$$
\xymatrix{\ldots \ar[r] & H^*(W) \ar[r] \ar[d] & SH^*_{<A}(W) \ar[r] \ar[d] & SH_{+,<A}^*(W) \ar[r]\ar[d] & H^{*+1}(W) \ar[r]\ar[d] & \ldots \\
	\ldots \ar[r] & H^*(V) \ar[r] & SH^*(V) \ar[r] & SH_+^*(V) \ar[r] & H^{*+1}(V) \ar[r] &  \ldots
 }
$$
If we also take the direct limit for $A$, then we get the Viterbo transfer map for the full symplectic cohomology.

Moreover, $SH^*(W)$ is a unital ring by the pair-of-pants construction \cite[\S 6]{ritter2013topological}. The continuation map $H^*(W)\to SH^*(W)$ as well as the (full) Viterbo transfer $SH^*(W)\to SH^*(V)$ are unital ring maps \cite[\S 9, \S 15]{ritter2013topological}. If we take  $SH^{[0],*}(W)$, then the Viterbo transfer map maps to the part of $SH^*(V)$ generated by orbits that are \emph{contractible in $W$}, which can then be included into the full symplectic cohomology generated by all orbits.

\subsection{$S^1$-equivariant symplectic cohomology}
The $S^1$ action on the free loop space gives rise to an $S^1$-equivariant analogue for symplectic cohomology \cite{bourgeois2017s,seidel2006biased}. More precisely, one can consider Hamiltonians with slope $a$ that are parameterized  over $ES^1=S^{\infty}$ (in the actual construction, we need to approximate $ES^1$ by $S^{2m+1}$). In particular, by counting parameterized Floer trajectories, we have maps $\delta^1,\delta^2,\ldots$ defined on the regular symplectic cochain complex, such that $\sum_{i+j=k} \delta^i\circ \delta^j=0$ for $k\ge 0$. Then $\delta^{S^1}=\sum_{k=0}^\infty u^k\delta^k$ defines a differential on $C(H)\otimes R[u,u^{-1}]/u$,  whose cohomology is denoted by $SH^*_{S^1,<a}(W)$. The full $S^1$-equivariant symplectic cohomology $SH^*_{S^1}(W)$ is the direct limit with respect to continuation maps for $a\to \infty$, see \cite{bourgeois2017s,gutt2017positive,MR4039816,zhou2019symplectic} for details. We also have the splitting into positive and zero parts, which induces the following long exact sequence of  $H^*(BS^1)=R[u]$-modules,
\begin{equation}\label{eqn:connect}
\ldots \to H^*_{S^1}(W)\to SH^*_{S^1}(W) \to SH^*_{+,S^1}(W)\stackrel{\delta}{\to} H^{*+1}_{S^1}(W)\to \ldots,
\end{equation}
where $H^*_{S^1}(W)=H^*(W)\otimes_{\Z} R[u,u^{-1}]/u$. 
\begin{definition}\label{def:SD}
	Let $\pi_0$ denote the projection $H^*_{S^1}(W)\to H^0(W)$,  we say $W$ carries a $k$-semi-dilation iff there exists $x\in SH^*_{+,S^1}(W)$, such that $\pi_0\circ \delta(x)=1$ and $u^{k+1}(x)=0$.
\end{definition}
It is clear from definition, that $k$-semi-dilations only depend on the part $SH^{[0],*}_{+,S^1}(W)$ generated by \emph{contractible} orbits.

\begin{remark}
	A symplectic dilation was introduced by Seidel-Solomon \cite{seidel2012symplectic} as an element $x\in SH^1(W)$, such that $\Delta(x)=1$, where $\Delta$ is the BV operator. In \cite{zhou2019symplectic}, we introduced the concept of $k$-dilations, an exact domain $W$ admits a $k$-dilation iff there exists $x\in SH^*_{+,S^1}(W)$, such that $\delta(x)=1$ and $u^{k+1}(x)=0$. Then the existence of a $1$-dilation is equivalent to the existence of a symplectic dilation \cite[Proposition 3.8]{zhou2019symplectic}, and the existence of a $k$-dilation for some $k$ is equivalent to the existence of a cyclic dilation for $h=1$ introduced in \cite{li2019exact}, see \cite[Remark 1.1]{zhou2019symplectic}. The existence of $k$-dilations implies the existence of $k$-semi-dilations \cite[Proposition 3.8]{zhou2019symplectic}.
\end{remark}
Let $V\subset W$ be an exact subdomain. By the family version of the construction in \S \ref{ss22}, we have the following Viterbo transfer map which respects the $R[u]$-module structure, see  \cite{gutt2017positive}.
$$
\xymatrix{\ldots \ar[r] & H_{S^1}^*(W) \ar[r] \ar[d] & SH^*_{S^1,<a}(W) \ar[r] \ar[d] & SH_{+,S^1,<a}^*(W) \ar[r]\ar[d] & H^{*+1}_{S^1}(W) \ar[r]\ar[d] & \ldots \\
	\ldots \ar[r] & H^*_{S^1}(V) \ar[r] & SH^*_{S^1}(V) \ar[r] & SH_{+,S^1}^*(V) \ar[r] & H^{*+1}_{S^1}(V) \ar[r] &  \ldots
}
$$
In particular, if $W$ carries a $k$-semi-dilation, then by the commutativity of rightmost square as $\R[u]$-modules, we have that $V$ also carries a $k$ semi-dilation.
\begin{example}
	The following examples have a $k$-semi-dilation for some $k$.
	\begin{enumerate}
		\item The cotangent bundle $T^*Q$ of a simply connected closed manifold $Q$ \cite[Proposition 5.1]{zhou2019symplectic} and more generally, plumbings of $\{T^*Q_i\}$ with respect to any tree for simply connected closed spin manifold $Q_i$ such that $\dim Q_i\ge 3$ \cite[Proposition 6.2]{li2019exact};
		\item The Brieskorn variety $x_0^n+\ldots+x_n^n=1$ \cite[Theorem A]{zhou2019symplectic};
		\item Lefschetz fibration whose fiber admits a $k$-semi-dilation \cite[Proposition 3.31]{zhou2019symplectic};
		\item Products of exact domains if one of them admits a $k$-semi-dilation, \cite[Proposition 3.30]{zhou2019symplectic};
		\item Flexible surgery does not affect the existence of $k$-semi-dilations under mild assumptions \cite[Proposition 3.32]{zhou2019symplectic}.
	\end{enumerate}
Using the functorial property of $k$-semi-dilations and Lefschetz fibrations, there are many more Brieskorn varieties with $k$-semi-dilations, when the Brieskorn singularity is canonical, see \cite[\S 5]{zhou2019symplectic}.
\end{example}

\begin{example}\label{ex:counter}
	Let $L$ be a smooth $K(\pi,1)$ space, then $T^*L$ does not admit a $k$-semi-dilation. This is because $SH_{S^1}(T^*L)$ is the equivariant homology of the free loop space possibly twisted by a local system \cite{abbondandolo2006floer, abouzaid2013symplectic, salamon2006floer,viterbo2018functors}, hence $SH^*_{+,S^1}(T^*L)$ is generated by non-contractible loops as $L$ is a $K(\pi,1)$ space. Therefore $T^*L$ can not support a $k$-semi-dilation for any $k$, as the class $x\in SH^*_{+,S^1}(W)$ for a $k$-semi-dilation is represented by contractible orbits.
\end{example}

\section{Truncated Viterbo transfer}
Let $L\subset W$ be a Lagrangian. We say $L$ is exact if $\lambda$ is exact on $L$. In this case, the complement of a Weinstein neighborhood of $L$ is an exact cobordism.  In general, $\lambda|_L$ is only a closed class, and the complement is a strong cobordism and the Viterbo transfer in \S \ref{ss22} fails. In this case, \eqref{eqn:Amin} becomes
$$A_{\min}(L,W)=\inf\left\{\int \gamma^*\lambda\left| \gamma \in \pi_1(L) \text{ that is contractible in } W, \int \gamma^*\lambda>0\right. \right\}.$$

\begin{proposition}\label{prop:Viterbo}
	Let $(W,\lambda)$ be an exact domain and $L$ a Lagrangian such that $A_{\min}(L,W)\ge 2A>0$, then we have a Viterbo transfer  $SH^*_{<A}(W) \to SH^*(T^*L)$ such that the following exact sequences commute
	$$
	\xymatrix{ \ldots \ar[r]& H^*(W) \ar[r] \ar[d] & SH^{[0],*}_{<A}(W) \ar[r] \ar[d] & SH^{[0],*}_{+,<A}(W) \ar[r] \ar[d] & H^{*+1}(W) \ar[r] \ar[d] & \ldots \\
		\ldots\ar[r] & H^*(T^*L) \ar[r] & SH^*(T^*L) \ar[r] & SH_{+}^*(T^*L) \ar[r] & H^{*+1}(T^*L) \ar[r] & \ldots }
	$$
\end{proposition}
\begin{proof}
	Fixing a metric on $L$, we have the unit cotangent bundle $(D^*L,\lambda_{std})$. We define $(D^*_{\epsilon}L,\lambda_{std})=D^*L\backslash (\partial D^*L\times (\epsilon,1])$. By the Weinstein neighborhood theorem, we have a symplectic embedding $(D_{\epsilon}^*L,\rd \lambda_{std})$ into $W$. WLOG,  we can assume there exits a closed one form $\beta \in \Omega^1(D^*_{\epsilon}L)$, which is a pullback from $\Omega^1(L)$, such that $\lambda=\lambda_{std}+\beta$ on $D^*_{\epsilon}L$. Then we consider Hamiltonians $H$ on $\widehat{W}$ in the definition of the Viterbo transfer that are $C^2$ small perturbations to functions with the following properties,
	\begin{enumerate}
		\item $0$ on $D^*_{\frac{\epsilon}{2}}L$,
		\item linear with slope $\frac{2A}{\epsilon}$ with respect to $r\in (0,\epsilon]$ on $D^*_{\epsilon}L\backslash D^*_{\frac{\epsilon}{2}}L$,\footnote{The standard $r_0$ coordinate for the completion $D^*_{\frac{\epsilon}{2}}L\cup \partial D^*_{\frac{\epsilon}{2}}L\times (1,\infty)_{r_0}=T^*L$ equals to $\frac{2}{\epsilon}r$ on $D^*_{\epsilon}L\backslash D^*_{\frac{\epsilon}{2}}L$. Therefore the slope w.r.t. the $r_0$ coordinate is $A$. }
		\item $A$ on $W\backslash D^*_{\epsilon}L$,
		\item linear with slope $A$ on $\partial W \times (1,\infty)$.
	\end{enumerate} 
     We write  $\Gamma:=A_{\min}(L,W)$, which by assumption is $\ge 2A$. As in \S \ref{ss22}, we have the following five types of orbits.
     \begin{enumerate}[(I)]
     	\item Constant orbits on  $D^*_{\frac{\epsilon}{2}}L$ with $\cA_H\approx 0$.
     	\item Non-constant periodic orbits near $\partial D^*_{\frac{\epsilon}{2}}L$. The symplectic action using $\lambda_{std}$ for those type \eqref{II} orbits is the same as before, i.e.\ in the region $(-A,0)$.\footnote{Translating back to the notation in \S \ref{ss22}, we have $\epsilon'=1,A'=A,B'=A$, where $\epsilon',A',B'$ are $\epsilon,A,B$ in \S \ref{ss22}. This also explains why the symplectic action using $\lambda_{std}$ in the next case is in $(-A,A)$.}  But the actual symplectic action using $\lambda$ differs by $\int \gamma^*\beta$ for the orbit $\gamma$. In particular, those orbits have action in $(-A,0)+\Z \Gamma$.
     	\item Non-constant periodic orbits near $\partial D^*_{\epsilon}L$.  The actual symplectic action using $\lambda$ again differs from that of type \eqref{III} orbits in \S \ref{ss22} by $\int \gamma^*\beta$ for the orbit $\gamma$. Hence the symplectic action is in  $(-A,A)+\Z \Gamma$. 
     	\item  Constant orbits on $W\backslash D^*_{\epsilon}L$, with $\cA_H\approx A$, as $H\approx A$ there.
     	\item Non-constant periodic orbits near $\partial W$ with $\cA_H\in (0, A)$.
     \end{enumerate}
     Then we consider the Hamiltonian $H_W$ on $\widehat{W}$, which is admissible with slope $A$. Since the symplectic action of orbits of $X_{H_W}$ is in $(-A,0)$ and the continuation map induced from the non-increasing homotopy from $H$ to $H_W$ increases symplectic action, we get a cochain map from $C^{[0]}(H_W)$ to the complex generated by orbits of $X_H$ that are contractible in $W$ with action in $(-A,0)$. 
     Since $\Gamma\ge 2A$, any type \eqref{II}, \eqref{III} orbit $\gamma$ with action in $(-A,0)$ has the property that $\int \gamma^*\beta=0$.  
     \begin{claim}
     	 Type \eqref{I} and \eqref{II} orbits with action in $(-A,0)$ generate a further quotient complex, which is isomorphic to $C^*_{\beta}(H_{D^*_{\epsilon}L})$, i.e.\ the subcomplex generated by orbits annihilated by $\beta$, where $H_{D^*_{\epsilon}L}\in C^{\infty}(S^1\times \widehat{D^*_{\epsilon}L})$ is the linear extension of $H|_{S^1\times D^*_{\frac{3\epsilon}{4}}L}$ to $S^1\times T^*L$.
     \end{claim}
     \begin{proof}[Proof of the claim]
     	 Let $x,y$ be two such orbits. We can require the almost complex structure to be cylindrically convex for $\lambda_{std}$ on near $\partial D^*_{\frac{3}{4}\epsilon}L$\footnote{This does not contradict the ``genericity" of $J$, as it is sufficient to perturb $J$ near orbits to achieve transversality.}, i.e.\ $\lambda_{std}\circ J=\rd \frac{2}{\epsilon}r$, we claim if $u\in \cM_{x,y}$ then $\Ima u \subset D^*_{\frac{3}{4}\epsilon}L$.  Assume otherwise that $u$ escapes $D_{\frac{3}{4}\epsilon}^*L$, and $S=u^{-1}(\widehat{W}\backslash D_{\frac{3}{4}\epsilon}^*L)$. WLOG, we may assume  $u\pitchfork \partial D^*_{\frac{3}{4}\epsilon}L$, hence $S$ is a compact surface with boundary equipped with complex structure $j$, then by homology reasons $u|_{\partial S}$ represents a class annihilated by $\beta$, in particular,  $\int_{\partial S} u^*\lambda = \int_{\partial S} u^*\lambda_{std}$, this allows us to apply the integrated maximum principle by Abouzaid-Seidel \cite{abouzaid2010open}, or \cite[Lemma 2.2]{cieliebak2018symplectic}. More precisely,  we have
     	\begin{eqnarray*}
     		0< E(u|_S) & := & \frac{1}{2}\int_S |\rd u-X_H\otimes \rd t|^2\rd \vol \\
     		& = & \int_S u^*\rd \lambda - u^*\rd H \wedge \rd t \\
     		&= &\int_S\rd(u^*\lambda-u^*H\rd t)\\
     		& = & \int_{\partial S} u^*\lambda-(u^*H)\rd t \\
     		& = & \int_{\partial S} u^*\lambda_{std}-(u^*H)\rd t \\
     		& = & \int_{\partial S} \lambda_{std}(\rd u-X_H(u)\otimes \rd t)\\
     		& = & \int_{\partial S} \lambda_{std}(J\circ (\rd u - X_H(u)\otimes \rd t) \circ(-j)) \\
     		& = & \int_{\partial S} \frac{2}{\epsilon}\rd r \circ \rd u \circ (-j)\le 0
     	\end{eqnarray*}  
     	which is a contradiction.  The second last equality follows from that $u$ solves the Floer equation $(\rd u -X_H(u)\otimes \rd t)^{0,1}=0$. The last equality follows from $\lambda_{std}\circ J=\rd \frac{2}{\epsilon}r$ and $\rd r (X_H)=0$ on $\partial D^*_{\frac{3}{4}\epsilon}L$. The last inequality follows from the fact that for each tangent vector $\xi$ of $\partial S$ defining its boundary orientation, $j\xi$ points into $S$, therefore $\rd r \circ \rd u (j\xi)\ge 0$. In other words, since we only consider orbits $\gamma$ with $\int \gamma^*\beta =0$, the integrated maximum principle can be applied as usual. As a consequence, combined with the asymptotic behavior lemma \cite[Lemma 2.3]{cieliebak2018symplectic},  there is no differential from a type \eqref{III} orbit to a type \eqref{II} or \eqref{I} orbit if both of them are annihilated by $\beta$. This shows that type \eqref{I} and \eqref{II} orbits with action in $(-A,0)$ generate a quotient complex.

     	By \cite[Lemma 2.2]{cieliebak2018symplectic}, all curves involved in the definition of Hamiltonian-Floer cohomology of the linear extension $H_{D^*_{\epsilon}L}$ on $T^*L$ is contained in $D^*_{\frac{3\epsilon}{4}}L$, where the geometric data is identified with that of $W$ restricted to $D^*_{\frac{3\epsilon}{4}}L$. In particular, we can identify the moduli spaces for the differential, hence our quotient complex is identified with $C^*_{\beta}(H_{D^*_{\epsilon}L})$ on the nose.
     \end{proof}
      Hence there is a truncated Viterbo transfer map $C^*(H_W)\to C^*_{\beta}(H_{D^*_{\epsilon}L})$ which respects the splitting into $C_0,C_+$. Then the claim follows from the composition with the continuation map to the full symplectic cohomology of $T^*L$.
\end{proof}

\begin{remark}\label{rmk:better_bound}
In fact, Proposition \ref{prop:Viterbo} holds for $A_{\min}(L,A)>A$. To see this, for $\delta\ll 1$, we consider the Hamiltonian with slope $\frac{A}{(1-\delta)\epsilon}$ on $D^*_{\epsilon} L\backslash D^*_{\delta \epsilon} L$, then type \eqref{II} orbits have action in $(-\frac{A \delta}{1-\delta},0)+\Z\Gamma$ and type \eqref{III} orbits have action in $(-\frac{A \delta}{1-\delta}, A)+\Z\Gamma$, then the argument in Proposition \ref{prop:Viterbo} can be applied if $\Gamma > A+\frac{A\delta}{1-\delta}$. Then the claim follows when $\delta \to 0$. This bound is sharp, which can be seen from a family of thin ellipsoids converging to the infinite cylinder $\D\times \C^{n-1}$, where $\D\subset \C$ is the unit disk. Another perspective of such bound is from the functoriality of SFT and the relation between SFT and string topology \cite{cieliebak2009role}, see Remark \ref{rmk:SFT}.
\end{remark}

By exactly the same argument, we have the following parameterized version for equivariant symplectic cohomology. 
\begin{proposition}\label{prop:Viterbo2}
	Let $W$ be an exact domain and $L$ a Lagrangian such that $A_{\min}(L,W)\ge 2A>0$, then we have a Viterbo transfer map $SH^{[0],*}_{S^1,<A}(W) \to SH_{S^1}^*(T^*L)$ of $R[u]$-modules, such that the following exact sequences commute
	$$
	\xymatrix{ \ldots \ar[r]& H_{S^1}^*(W) \ar[r] \ar[d] & SH^{[0],*}_{S^1,<A}(W) \ar[r] \ar[d] & SH^{[0],*}_{+,S^1,<A}(W) \ar[r] \ar[d] & H^{*+1}_{S^1}(W) \ar[r] \ar[d] & \ldots \\
		\ldots\ar[r] & H^*_{S^1}(T^*L) \ar[r] & SH^*_{S^1}(T^*L) \ar[r] & SH_{+,S^1}^*(T^*L) \ar[r] & H^{*+1}_{S^1}(T^*L) \ar[r] & \ldots }
	$$
\end{proposition}

Next, we consider general symplectically aspherical domains. Let $H_W$ be an admissible Hamiltonian on $W$ with slope $a$, the symplectic action of a contractible orbit $\gamma$ is given by \eqref{eqn:action} if $\gamma$ is contractible in $\partial W$. If $\gamma$ is only contractible in $W$, then \eqref{eqn:action}, \eqref{eqn:action'} might differ\footnote{Since our non-constant orbits are in $\partial W \times (1,\infty)$, where $\lambda$ is still defined even though $\omega$ is not globally exact.}. Since for a Hamiltonian of slope $a$, there are only finitely many orbits. Therefore we can find $f(a)\ge a>0$, such that symplectic action of periodic orbits of $X_{H_W}$ is supported in $(-f(a),f(a))$. 
\begin{proposition}\label{prop:Viterbo3}
	Let $W$ be a symplectically aspherical domain and $L$ a Lagrangian such that $A_{\min}(L,W)\ge 4f(A)>0$, then we have a Viterbo transfer  $SH^{[0],*}_{<A}(W) \to SH^*(T^*L)$ such that the following commutes
	$$
	\xymatrix{ H^*(W) \ar[r]^{\iota} \ar[d] & SH^{[0],*}_{<A}(W)\ar[d]  \\
		  H^*(T^*L) \ar[r] & SH^*(T^*L) }
	$$
\end{proposition}
\begin{proof}
      By the Weinstein neighborhood theorem, we have a symplectic embedding $(D_{\epsilon}^*L,\rd \lambda_{std})$ into $W$ as before. Then we consider Hamiltonians $H$ on $\widehat{W}$ in the definition of the Viterbo transfer that are $C^2$ perturbations to functions with the following properties,
	\begin{enumerate}
		\item $0$ on $D^*_{\frac{\epsilon}{2}}L$,
		\item linear with slope $\frac{4f(A)}{\epsilon}$ with respect to $r\in (0,\epsilon]$ on $D^*_{\epsilon}L\backslash D^*_{\frac{\epsilon}{2}}L$.
		\item $2f(A)$ on $W\backslash D^*_{\epsilon}L$,
		\item linear with slope $A$ on $\partial W \times (1,\infty)$.
	\end{enumerate} 
      We write  $\Gamma:=A_{\min}(L,W)$, which by assumption is $\ge 4f(A)$. As in \S \ref{ss22}, we have the following five types of orbits.
    \begin{enumerate}[(I)]
    	\item Constant orbits on  $D^*_{\frac{\epsilon}{2}}L$ with $\cA_H\approx 0$.
    	\item Non-constant periodic orbits near $\partial D^*_{\frac{\epsilon}{2}}L$. The symplectic action using $\lambda_{std}$ for those type \eqref{II} is the same as before, i.e.\ in the region $(-2f(A),0)$.  The actual symplectic action using $\omega$ for a contractible orbit $\gamma$ of type \eqref{II} differs from the symplectic action using $\lambda_{std}$ by $\int u^*\omega$, where $u\in \pi_2(W,L)$ with boundary homotopic to $\gamma$ in $D_{\epsilon}^*L$. To see this, let $u_0:D\to W$ to be the disk used to compute the symplectic action of $\gamma$ in \eqref{eqn:action'}, and $u_1$ be the annulus in $D^*_{\epsilon}L$ with $\partial u_1=-\gamma \cup \eta$, where $\eta\subset L$. Then the glued disk $u_0\# u_1$ is in $\pi_2(W,L)$. Then we have
    	\begin{eqnarray*}
    	\int (u_0\# u_1)^*\omega  & =  & \int u_0^*\omega +\int u_1^*\omega = \int u_0^*\omega + \int u_1^*(\rd \lambda_{std}) \\
    	&  = & \int u_0^*\omega -\int \gamma^*\lambda_{std}+\int \eta^*\lambda_{std}= \int u_0^*\omega -\int \gamma^*\lambda_{std},
    	\end{eqnarray*}
    	i.e.\ the action difference, where the last equality follows from that $\lambda_{std}|_L=0$. Therefore the actual symplectic action is in  $(-2f(A),0)+\Z \Gamma$. It is also important to note that the integral $\int u^*\omega$ is independent of $u$, only depends on $\gamma$, as $W$ is symplectically aspherical and $\pi_1(L)\to \pi_1(D^*_{\epsilon}L)$ is an isomorphism.
    	\item Non-constant periodic orbits near $\partial D^*_{\epsilon}L$.  The actual symplectic action using $\lambda$ again differs from that of type \eqref{III} in \S \ref{ss22} by $\int u^*\omega$. Hence the symplectic action is in   $(-2f(A),2f(A))+\Z \Gamma$. 
    	\item  Constant orbits on $W\backslash D^*_{\epsilon}L$, with $\cA_H\approx 2f(A)$, as $H\approx 2f(A)$ there.
    	\item Non-constant periodic orbits near $\partial W$ with $\cA_H\in (-f(A), f(A))+2f(A) = (f(A),3f(A))$.
    \end{enumerate}
      Then we consider the Hamiltonian $H_W$ on $\widehat{W}$, which is admissible with slope $A$. Since the symplectic action of orbits of $X_{H_W}$ is in $(-f(A),f(A))$ and the continuation map $C^{[0],*}(H_W)\to C^{[0],*}(H)$ increases symplectic action, we get a cochain map from $C^{[0],*}(H_W)$ to the complex generated by orbits of $X_H$ that are contractible in $W$ with action in $(-f(A),f(A))$.  Since $\Gamma\ge 4f(A)$, any type \eqref{II}, \eqref{III} orbit in this action window has the property that $\int u^*\omega=0$, where $u\in \pi_2(W,L)$ is the disk as before. We claim that type \eqref{I} and \eqref{II} orbits with action in $(-f(A),f(A))$ generate a quotient complex, which is isomorphic to $C^*_{\omega}(H_{D^*_{\epsilon}L})$, i.e.\ the subcomplex generated by orbits in the homotopy classes that are contractible in $W$ and $\int u^*\omega=0$ for the contracting disk $u$, where $H_{D^*_{\epsilon}L}\in C^{\infty}(S^1\times \widehat{D^*_{\epsilon}L})$ is the linear extension of the truncation.  In view of the proof of Proposition \ref{prop:Viterbo}, this boils down to the integrated maximum principle for type \eqref{II}, \eqref{III} orbits with $\int u^*\omega=0$. More precisely, it is sufficient to prove that $\int_S u^*\omega = \int_{\partial S} \lambda_{std}$, where $S\subset \R\times S^1$ is the part where $u$ exceeds the level set to which we try to apply the integrated maximum principle. Let $u_x,u_y$ be the capping disks used to define symplectic action \eqref{eqn:action'} for $x$ and $y$. Therefore we have 
      $$\int_{\R\times S^1} u^*\omega=\int_D (u_x^*\omega -u_y^*\omega).$$
      Since $x,y$ have the property that $\int u_x^*\omega = \int x^*\lambda_{std}$ and  $\int u_y^*\omega = \int y^*\lambda_{std}$, we have 
      $$\int_{\R\times S^1} u^*\omega=\int x^*\lambda_{std}-y^*\lambda_{std}.$$
      Since $u|_{\R\times S^1\backslash S}$ is contained in the Weinstein neighborhood of $L$, we have, by Stokes' theorem,
      $$\int_{\R\times S^1\backslash S} u^*\omega = \int_{\R\times S^1\backslash S} u^*(\rd \lambda_{std})= \int \left(x^*\lambda_{std}-y^*\lambda_{std}\right)-\int_{\partial S} \lambda_{std}.$$
      Therefore we have  $\int_S u^*\omega = \int_{\partial S} \lambda_{std}$. As a consequence, the claim holds, as well as the truncated Viterbo transfer map.
\end{proof}

\begin{proof}[Proof of Theorem \ref{thm:main}]
	If $W$ is a symplectically aspherical domain such that $SH^*(W)=0$, i.e.\ $1=0$ in $SH^{[0],*}(W)$, then $H^*(W)\stackrel{\iota}{\to} SH^{[0],*}(W)$ maps $1$ to $0$. Since $SH^{[0,*]}(W)=\varinjlim SH^{[0],*}_{\le a}(W)$, then by definition, 
	$$A:=\min\left\{ a \left| \iota(1)=0 \in SH^{[0],*}_{<a}(W), 1\in H^0(W) \right. \right\}$$
	is finite. Then we have  $A_{\min}(L,W)< 4 f(A)$. Assume otherwise, by Proposition \ref{prop:Viterbo3}, we have that $1\in H^*(T^*L)$ is also mapped to zero in $SH^*(T^*L)$, contradicting that $SH^*(T^*L)\ne 0$ \cite{abbondandolo2006floer,abouzaid2013symplectic,viterbo2018functors,salamon2006floer}. This proves Claim (1).
	
	If $W$ is exact and admits a $k$-semi-dilation, then we have 
		$$B:=\min\left\{ a \left| \exists x \in SH^*_{+,S^1,<a}(W), \text{s.t.\ } u^{k+1}(x)=0, \pi_0\circ \delta(x)=1\in  H^0(W)\right. \right\}$$
	is finite. If there is a $K(\pi,1)$ Lagrangian $L\subset W$ with $A_{\min}(L,W) \ge  2B$, then by Proposition \ref{prop:Viterbo2}, we have a commutative square of $R[u]$-modules
	$$
	\xymatrix{SH^{[0],*}_{+,S^1,<2B}(W) \ar[r]^{\delta} \ar[d]^{\phi} & H^{*+1}_{S^1}(W)  \ar[d] \ar[r]^{\pi_0} & H^0(W)\ar[d] \\
	SH_{+,S^1}^*(T^*L) \ar[r]^{\delta} & H^{*+1}_{S^1}(T^*L) \ar[r]^{\pi_0} & H^0(T^*L)}
	$$
	Since there is a $x\in SH^{[0],*}_{+,S^1,<2B}(W)$ such that $\pi_0\circ \delta(x)=1$ and $u^{k+1}(x)=0$, then $\pi_0\circ \delta(\phi(x))=1$ and $u^{k+1}(\phi(x))=0$, i.e.\ $T^*L$ admits a $k$-semi-dilation, contradicting Example \ref{ex:counter}. This proves Claim (2).
\end{proof}

\begin{remark}
Following from Remark \ref{rmk:better_bound}, we have $A_{\min}(L,W)\le A$ when $W$ is an exact domain with $SH^*(W)=0$. When $W$ is given by the ellipsoid $\sum_{i=1}^n \frac{|z_i|^2}{a_i}\le 1$ for $0<a_1\le \ldots \le a_n$, one can show that $A=\pi a_1$.\footnote{Assume for simplicity $a_1<a_2$, then the non-degenerate Reeb orbit in the $z_1$-coordinate plane is the only Reeb orbit to have the right Conley-Zehnder index to kill $1$.} Then it contains a Lagrangian torus $L=\{|z_i|^2=b_i\}$, where $b_1=a_1-\epsilon$ for $0< \epsilon \ll 1$ and $b_i=\frac{a_i\epsilon}{a_1}$ for $i>1$. Then we have $A_{\min}(L,W)=\min \{\pi b_i\}$, which is $\pi b_1=\pi a_1-\pi \epsilon$ when $a_i\gg 0$ for $i>1$. That is the bound  $A_{\min}(L,W)\le A$ is close to a sharp bound when $W$ is a thin ellipsoid. However the bound is much weaker when $W$ is a round ball, i.e.\ $a_i=1$, where the sharp bound for Lagrangian torus is $\frac{\pi}{n}$  by Cieliebak-Mohnke \cite{cieliebak2018punctured}.
\end{remark}

\begin{proof}[Proof of Theorem \ref{thm:local}]
	By \cite{albers2017local} if $\pi_2(Q)\to H_2(Q)$ is nontrivial, then there exists a local system $\rho$ on $\cL_0Q$ (the space of contractible loops), such that $SH^*(T^*Q;\rho)=0$.  As the presence of local system only twists the formula for the differential but uses the same moduli spaces, Proposition \ref{prop:Viterbo} holds for symplectic cohomology with local systems by the same argument. Such local system is classified by $\Hom_{inv}(\pi_2(W),\C^*)$, i.e.\ the subset of $\Hom(\pi_2(W),\C^*)$ that is invariant under the $\pi_1(W)$ conjugation. Since $\pi_2(L)\to \pi_2(T^*Q)$ is trivial, we have the local system is trivial on $D^*_{\epsilon}L$, hence $SH^*(D^*_{\epsilon}L; \rho|_{D^*_{\epsilon}L} )=SH^*(T^*L;\C)\ne 0$. Then the proof of Theorem \ref{thm:main} can be applied to finish the proof.
\end{proof}

\section{The chord conjecture}
For the completeness, we recall the proof of Theorem \ref{thm:chord} from \cite{mohnke2001holomorphic}.
\begin{proof}[Mohnke's proof of Theorem \ref{thm:chord}] 
	Let $\alpha$ be a contact form on $Y$, which is, a priori, different from the contact form $\lambda|_{\partial W}$. However, we have $\alpha = f\lambda|_{\partial W}$ for some positive function $f$ on $Y$. Then $\widehat{\lambda}$ restricted the contact hypersurface $r=f(x)$ in $\widehat{W}$ is $\alpha$ (through the obvious identification of $r=f(x)$ and $Y=\partial W$). Note that we can find $r_0>0$, such that $r_0f\le 1$, i.e.\ the contact hypersurface $r=r_0f$ is contained in the domain $W$ and   $\widehat{\lambda}$ restricted the contact hypersurface $r=r_0f(x)$ in $W$ is $r_0\alpha$. Since the Reeb flows of $r_0\alpha$ and $\alpha$ are the same up to rescaling,  the chord conjecture holds for $\alpha$ iff it holds for $r_0\alpha$. Let $W'$ denote the exact subdomain bounded by $r=r_0f$. Since $W'\subset W$ is an homotopy equivalence, we have $\cA_{\min}(\cL_{all},W')\le \cA_{\min}(\cL_{all},W) <\infty$. Let $\phi_t$ denote the Reeb flow on $\partial W'$, i.e.\ Reeb flow of $r_0\alpha$. If the Legendrian $\Lambda$ has no Reeb chord, then $\Phi:\Lambda \times [\epsilon,1]\times \R \to \partial (W'\times [\epsilon,1]_{r'}, \rd(r'\lambda|_{\partial W})), (x,s,t)\mapsto (\phi_t(x),s)$ is an embedding. Here $r'$ is the cylindrical coordinate for the symplectization of $\partial W'$, which is different from the $r$ coordinate on $\widehat{W}$. Then it is direct to check that any embedded loop $\gamma \in [\epsilon,1]\times \R$ gives rise to a Lagrangian $\Phi(\Lambda \times \gamma)\subset \partial W'\times [\epsilon,1]\subset W'$. Note that $\lambda|_{\Lambda}=0$ and $\Phi^*(\rd(r'\lambda|_{\partial W'}))=\rd s\wedge \rd t$, then we know that $A_{\min}(\Phi(\Lambda \times \gamma),W')$ is the the area enclosed by $\gamma$, which can be arbitrarily big, contradiction! The situation for $K(\pi,1)$ Legendrians is similar as the product of a $K(\pi,1)$ space with $S^1$ is again a $K(\pi,1)$ space.
\end{proof}

\begin{proof}[Proof of Corollary \ref{cor:chord}]
	    Cases \eqref{ch1}, \eqref{ch3} follow Theorem \ref{thm:chord} and \ref{thm:main}. For case \eqref{ch4}, if $\pi_2(\Lambda)\to \pi_2(T^*Q)$ is trivial, then $\pi_2(\Phi(\Lambda \times \gamma ))\to \pi_2(T^*Q)$ is also trivial. Then the existence of Reeb chords follows from Theorem \ref{thm:local} and the same argument of Theorem \ref{thm:chord}. In the case \eqref{ch2}, if $\pi_1(\Lambda)\to \pi_1(W)$ is trivial, then we know that $A_{\min}(\Phi(\Lambda \times \gamma),W)$ is again area enclosed by $\gamma$\footnote{In general, it might be smaller contributed by a loop with a nontrivial $\pi_1(\Lambda)$ component that is contractible only in $W$.}. Then the claim follows from the argument of Theorem \ref{thm:chord} and Theorem \ref{thm:main}.
\end{proof}

\begin{proof}[Proof of Corollary \ref{cor:homo}]
The proof follows from the same argument as \cite[Theorem 1.1]{MR2418284}, where we replace \cite{chekanov1998lagrangian} by Theorem \ref{thm:main} in the proof of \cite[Theorem 2.1]{MR2418284} to obtain a universal upper bound of the minimal symplectic area. 
\end{proof}

\begin{remark}\label{rmk:surgery}
	In the special case of $SH^*(W)=0$,  the chord conjecture for Legendrian spheres can also be proven by surgery following the idea in \cite{cieliebak2002handle}.  Let $\widetilde{W}$ be the domain obtained from attaching a Weinstein handle to the Legendrian sphere $\Lambda$. If $\Lambda$ has no Reeb chords, then by \cite{cieliebak2002handle}, we have the Viterbo transfer map $SH_+^*(\widetilde{W})\to SH_+^*(W)$ is an isomorphism. Note that we have the following commutative diagram,
	$$
	\xymatrix{ \ldots \ar[r] & H^*(\tilde{W}) \ar[r]\ar[d] & SH^*(\widetilde{W}) \ar[r]\ar[d] & SH_+^*(\widetilde{W})\ar[r]\ar[d] & H^{*+1}(\widetilde{W})\ar[r]\ar[d] & \ldots \\
		\ldots \ar[r] & H^*(W) \ar[r] & SH^*(W) \ar[r] & SH_+^*(W) \ar[r] & H^{*+1}(W) \ar[r] & \ldots}
	$$ 
	Since $SH^*(W)=0$, we have $SH_+^{*}(W)\to H^{*+1}(W)$ hits $1$. Therefore $SH_+^{*}(\widetilde{W})\to H^{*+1}(\widetilde{W})$ also hits $1$, because  $SH_+^*(\widetilde{W})\to SH_+^*(W)$ is an isomorphism. As a consequence, we have $H^{*+1}(W)\simeq SH_+^*(W)\simeq SH_+^{*}(\widetilde{W})\simeq H^{*+1}(\widetilde{W})$. However, we can not have $H^*(W)\simeq H^{*}(\widetilde{W})$, as they have different Euler characteristics,  which is a contradiction.
\end{remark}

\begin{remark}\label{rmk:SFT}
	Following the philosophy of SFT, if $V\subset W$ is not an exact subdomain, i.e.\ the cobordism $W\backslash V$ is not exact, the strong cobordism $W\backslash V$ gives rise to a Maurer-Cartan element by counting holomorphic caps in $W\backslash V$ \cite{cieliebak2009role}, and the Viterbo transfer map should hold for the symplectic cohomology of $V$ deformed by the Maurer-Cartan element. Then the classical Viterbo transfer map should hold if we truncated at an action smaller than the minimal action of the Maurer-Cartan element. In the case of $V=D^{*}_{\epsilon}L$, the action of a Maurer-Cartan element is a multiple of $A_{\min}(L,W)$ minus the period of a geodesic loop for metric $\epsilon g$. In particular, it is not obvious that the Maurer-Cartan action is greater than $A_{\min}(L,W)$. However, if we let $\epsilon\to 0$, we can replace the symplectic cohomology of $T^*L$ by the string topology of $L$ following \cite{cieliebak2018symplectic}. In this case, the Maurer-Cartan element is represented by the evaluation of boundaries of holomorphic disks, whose action is indeed bounded below by $A_{\min}(L,W)$. Those two constructions (deformed symplectic cohomology, the relation between SFT and string topology) faces various technical difficulties, whose full constructions with complete analytical details are not available in literature. The argument in this paper can be viewed as a substitute of such ideas in the classical constructions. 
\end{remark}

The closely related Weinstein conjecture asserts the existence of Reeb orbits on every compact contact manifold. It was proven in dimension $3$ by Taubes \cite{taubes2007seiberg} and various other cases \cite{abbas2005weinstein,acu2018planarity,hofer1993pseudoholomorphic}. All the proofs are based on existence of holomorphic curves, often abundances of holomorphic curves\footnote{In the context of ECH, the non-triviality of the $U$-map implies the abundance of holomorphic curves.}. The relation between holomorphic curves and Reeb orbits are much better understood than the following question. 
\begin{question}
	How abundances of holomorphic curves (in the symplectic domain) can be used to show the existence of Reeb chords for any Legendrian.
\end{question}
\begin{remark}
	Note that the conditions in Corollary \ref{cor:chord} imply that the symplectic manifold is uniruled \cite[Theorem 3.27]{zhou2019symplectic}. 
\end{remark}

\section{New examples with $k$-semi-dilations}
Let $X$ be a degree $m$ smooth hypersurface in $\CP^{n+1}$ for $m\le n$. By \cite[Lemma 7.1]{RSFT}, for any holomorphic section $s$ of $\cO(k)$, we have the affine variety $X\backslash s^{-1}(0)$ embeds exactly into $X\backslash S^{-1}(0)$, where $S$ is a holomorphic section of $\cO(k)$ such that $S^{-1}(0)$ is a smooth divisor with multiplicity $1$. In our case, we can assume $S^{-1}(0)$ in the intersection of $X$ with a generic smooth degree $k$ hypersurface in $\CP^{n+1}$. In view of the functorial property of $k$-semi-dilations, i.e.\ when $V\subset W$ is an exact subdomain and $W$ admits a $k$-semi-dilation, then so does $V$, Theorem \ref{thm:SD} is implied by the following.

\begin{theorem}\label{thm:SD'}
	If $k\le n+1-m$, then $X\backslash S^{-1}(0)$ admits a $k+m-2$-dilation.
\end{theorem}
To prove this theorem, we will adopt the method in \cite{diogo2019symplectic} as in \cite{zhou2019symplectic}. We first recall the setups. We use $D$ to denote the smooth divisor $S^{-1}(0)$ and $\check{X}$ to denote the exact (Weinstein) domain $X\backslash D$. Then the contact boundary $Y:=\partial \check{X}$ is equipped with a Boothby-Wang contact form \cite{MR112160} such that the Reeb flow is the $S^1$ action on the circle bundle with periods $2\pi\cdot \N$. Then we pick two Morse functions $f_D$ on $D$ and $f_{\check{X}}$ on $\check{X}$, such that $\partial_r f_{\check{X}}>0$ near $\partial \check{X}=Y$. We use $\Crit(f)$ to denote the set of critical points. We may assume $f_D$ is perfect. This is obvious when $n=2$. When $n\ge 4$, this follows from Smale's simplification of the handle representation since $D$ is simply connected and $H_*(D)$ is free by the Lefschetz hyperplane theorem. When $n=3$, we only need to consider $m=1,k=1,2,3$, $m=2,k=1,2$ and $m=3,k=1$, the existence of perfect Morse functions follows from \cite{harer1978handlebody} except for the case $m=k=2$. When $m=k=2$, $D$ is a $K3$ surface, hence there exists a perfect Morse function on $D$. We also assume $f_{\check{X}}$ has a unique minimum $m$. In the following, we will consider the filtered cochain complex for period up to $2\pi k$.
\begin{proposition}
	 The circles with wrapping number around $D$ from $1,\ldots,k$ are in different homology classes of $\check{X}$. 
\end{proposition}	
\begin{proof}
	 To see this, when $n\ge 3$, we have $D$ is simply connected by the  Lefschetz hyperplane theorem.   $\partial \check{X}$ is an $S^1$-bundle over $D$, whose first Chern class is the restriction of $k[H]$ by the adjunction formula, where $[H]\in H^2(\CP^{n+1})$ is the hyperplane class, i.e.\ the positive generator of $H^2(\CP^{n+1})$, which is the Poincar\'e dual to the homology class represented by a complex hyperplane $H$. When $n\ge 4$, $[H]$ is also a generator of $H^2(D)$ by the  Lefschetz hyperplane theorem. Then the Gysin exact sequence implies that $H^1(\partial \check{X})=0$ and the torsion part of $H^2(\partial \check{X})$ is $\Z/k\Z$. This implies, by the universal coefficient theorem, that $H_1(\partial\check{X})=\Z/k\Z$, which is generated by the fiber circle. When $n=3$, assume $[H]\in H^2(D)$ is divisible by $M\in \N$ and $M$ is the maximum of such divisors\footnote{If $m_1,m_2$ divide $[H]$, then the lowest common multiple of $m_1,m_2$ also divides $[H]$}. That $H^2(D)$ is free implies that the torsion of $H^2(\partial \check{X})=H^2(D)/\la k[H]\ra$ is $\Z/Mk\Z$, hence we can still conclude that those multiple covers of fibers are in different homology classes in $H_1(\partial \check{X})=\Z/Mk\Z$. When $n=2$, i.e.\ $m=1,k=1,2$ and $m=2,k=1$, the only non-trivial case is when $m=1,k=2$, which follows from \cite[Proposition 2.3]{MR2342909}\footnote{Alternatively, in this case, we have $D=\CP^1$ and $\partial \check{X}$ is the circle bundle associated to $\mathcal{\cO}(4)$ and $\check{X}=T^*\mathbb{RP}^2$. Then the claim follows from a direct check.}.
\end{proof}	
 In particular, there will be no Floer differential between orbits with different wrapping numbers around $D$ when the period is at most $2\pi k$. Note that $k\le n+1-m$ is exactly the monotonicity condition in \cite{diogo2019symplectic}. Next we list the properties of the symplectic cochain complex and its $S^1$-equivariant analogue from \cite{diogo2019symplectic} as follows. 
\begin{enumerate}
	\item\label{p1} The symplectic cochain complex up to period $2\pi k$ is generated by $x\in \Crit(f_{\check{X}})$ and $\check{p}_i,\hat{p}_i$ for $p\in \Crit(f_D)$ and $1\le i \le k$. The positive cochain complex is generated by $\check{p}_i,\hat{p}_i$ for $p\in \Crit(f_D)$ and $1\le i \le k$. Here $\check{p}_i,\hat{p}_i$ can be viewed as the two Hamiltonian orbits corresponding to the Reeb orbit that is the $i$-mutiple cover of the simple Reeb orbit corresponding to the critical point $p$.
	\item\label{p2} For $p,q\in \Crit(f_D)$, we have $\la \delta^0(\check{p}_i),\hat{q}_i\ra=\la c_1(\cO(k), [W^s_D(q)]\cap [W^u_D(p)])$, where $$W^u_D(p):=\{z\in D|\lim_{t\to \infty}\phi_t(z)=p\}$$ with $\phi_t$ is the 
	gradient flow of $f_D$, i.e.\ the unstable manifold of the negative gradient flow. Similarly, we have the stable manifold 
	$$W^s_D(p):=\{z\in D|\lim_{t\to -\infty}\phi_t(z)=p\}.$$
	as well as stable/unstable manifolds $W^s_{\check{X}}(x),W^u_{\check{X}}(x)$ for the Morse function $f_{\check{X}}$ and $x\in \Crit(f_{\check{X}})$.
	\item\label{p3} For $x\in \Crit(f_{\check{X}})$ such that $W^u_{\check{X}}(x)$ represents a homology class in $H_*(\check{X})$, we have $$\la \delta^0(\check{p}_k),x\ra=k \mathrm{GW}^{X,D}_{0,1,(k),A}([W^u_{\check{X}}(x)], [W^u_D(p)])$$ 
	Here $\mathrm{GW}^{X,D}_{0,1,(k),A}([W^u_{\check{X}}(x)],[W^u_D(p)])$ stands for the relative Gromov-Witten invariant counting rational curves in $X$ of the positive generator\footnote{That is $A$ maps to the positive generator of $H_2(\CP^{n+1})$ by $X\to \CP^{n+1}$. Note that when $n=2,m=2$, $X=\CP^1\times \CP^1$ and there are two classes in $H_2(X)$ with such property, then  $\mathrm{GW}^{X,D}_{0,1,(k),A}([m],[W^u_D(p)])$ should be understood as the sum of relative Gromov-Witten invariants of both classes.} $A\in H^2(X)$ with two marked points mapped to $W^u_{\check{X}}(x)$ and  $W^u_D(p)$ in $D$ with the intersection multiplicity with $D$ at least $k$. Since we only care about $\la \delta^0(\check{p}_k),m\ra$ for the unique minimum $m$, the formula above suffices for our purpose of studying $l$-semi-dilations.
	\item\label{p4} There are also Morse differentials between critical points of $f_{\check{X}}$, and from $\check{p}_k$ to other critical points of $f_{\check{X}}$. But they are not relevant for studying $l$-semi-dilations. Those four cases are all the terms of $\delta^0$. This is because $f_D$ is perfect and $p_1,\ldots,p_k$ are in different homology classes in $\check{X}$. 
	\item\label{p5} We have $\delta^1(\check{p}_i)=i\hat{p}_i$ and $\delta^i=0$ for $i\ge 2$. This follows from the $S^1$-equivariant transversality argument in \cite[Proposition 5.9, Remark 5.10]{zhou2019symplectic}, where the $S^1$-equivariant transversality is provided by \cite{diogo2019symplectic} as the monotonicity assumption holds. The coefficient being $i$ follows from \cite[Lemma 3.1]{bourgeois2017s}. 
	\item $\delta^i$ respects the splitting of $C_0,C_+$, i.e.\ $\delta^i$ has a decomposition into $\delta^i_0+\delta^i_++\delta^+_0$ with $\delta^i_0:C_0\to C_0, \delta^i_+:C_+\to C_+, \delta^i_{+,0}:C_+\to C_0$. The connecting map  $\delta$ in \eqref{eqn:connect} is defined by $\sum_{i=0}^\infty u^i \delta^i_{+,0}$.
\end{enumerate}

\begin{proof}[Proof of Theorem \ref{thm:SD'}]
	By \cite[Theorem 2.6]{gathmann2002absolute}, we have 
	$$\mathrm{GW}^{X,D}_{0,1,(k),A}([pt], [D]\cap^{n+1-m-k}[H])\ne 0,$$
	where $H\in H_{2n}(\CP^{n+1})$ is the hyperplane class. We assume $p^0$ is the critical point of $f_D$ of index $2(m+k-2)$, such that $[W^u_D(p^0)]=[D]\cap^{n+1-m-k}[H]$\footnote{In general, it is possible that $[D]\cap^{n+1-m-k}[H]$ is represented by a linear combination of critical points, we use a single one with coefficient $1$ for simplicity.}. Similarly we use $p^i$ to represent the critical point such that $[W^u_D(p^i)]=[D]\cap^{n+1-m-k+i}[H]$ for $1\le i \le m+k-2$. 
	We claim that $\eta:=\check{p}_k+\sum_{i=1}^{m+k-2}(-1)^iu^{-i} \check{p}^i_k$ represents a closed cochain in the positive $S^1$-equivariant cochain complex. To see this, we compute
	\begin{eqnarray*}
		\delta^{S^1}(\eta)&=& \delta^0_+(\check{p}^0_k)+\sum_{i=1}^{m+k-2}\sum_{j=0}^i (-1)^iu^{j-i}\delta^j_+(\check{p}^i_k) \\
		& \stackrel{\eqref{p5}}{=} & \delta^0_+(\check{p}^0_k)+\sum_{i=1}^{m+k-2}\sum_{j=0}^1 (-1)^iu^{j-i}\delta^j_+(\check{p}^i_k) \\
		& = & \sum_{i=0}^{m+k-3} (-1)^i u^{-i} (\delta_+^0(\check{p}_k^i)-\delta^1_+(\check{p}_k^{i+1}))+(-1)^{m+k-2}u^{2-m-k}\delta_+^0(\check{p}_k^{m+k-2})\\
		& \stackrel{\eqref{p2},\eqref{p5}}{=} & \sum_{i=0}^{m+k-3}(-1)^iu^{-i}(k\hat{p}^{i+1}_k-k\hat{p}^{i+1}_k) = 0
	\end{eqnarray*}
	One the other hand, we compute
	\begin{eqnarray*}
	 \pi_0\circ \delta(\eta) & = & \pi_0\circ \left( \delta^0_{+,0}(\check{p}^0_k)+\sum_{i=1}^{m+k-2}\sum_{j=0}^i (-1)^iu^{j-i}\delta^j_{+,0}(\check{p}^i_k)\right) \\
	 & \stackrel{\eqref{p3},\eqref{p5}}{=} & k\mathrm{GW}^{X,D}_{0,1,(k),A}([pt], [D]\cap^{n+1-m-k}[H])\ne 0
	\end{eqnarray*}
	Finally, $u^{m+k-1}(\eta) =0$ by definition. Therefore $\check{X}$ admits a $m+k-2$-dilation.
\end{proof}

\begin{remark}
	Following the fibration argument in \cite[\S 5]{zhou2019symplectic}, we can prove that the order of semi-dilation for $\check{X}$ is exactly $k+m-2$. It is natural to expect that examples in Theorem \ref{thm:SD} carry a $k+m-2$-dilation instead of a $k+m-2$-semi-dilation. However, this requires checking the vanishing of many other relative Gromov-Witten invariants like \cite[Theorem A]{zhou2019symplectic}, i.e.\ the differentials in \eqref{p4} is no longer irrelevant. 
\end{remark}

\begin{remark}
	One can derive a formula of the positive $S^1$-equivariant symplectic cohomology of smooth divisor complements similar to \cite{diogo2019symplectic}. Under the monotonicity assumption, we can use the $S^1$-equivariant transversality to apply a quotient construction of the positive $S^1$-equivariant symplectic cohomology instead of the Borel construction used in this paper, see \cite{bourgeois2017s}. In this case, the cochain complex is generated by $p_k$ for $p\in \Crit(f_D)$ and $k\ge 1$, with differential purely determined by Morse differential of $f_D$. All the other contributions from check orbits to hat orbits in \cite[Theorem 9.1]{diogo2019symplectic} now contributes to the $R[u]$-module structure. Note that by the Gysin exact sequence \cite{bourgeois2013gysin}, the positive $S^1$-equivariant symplectic cohomology as a $R[u]$-module can recover the regular positive symplectic cohomology. In particular, this is just a reformulation of \cite[Theorem 9.1]{diogo2019symplectic}. When the monotonicity assumption fails, we may not have a simple collections of degeneration as in \cite{diogo2019symplectic}, but the other consequence of monotonicity, namely $S^1$-equivariant transversality, is available through polyfold techniques \cite{zhou2020quotient}. In particular, it seems that positive $S^1$-equivariant symplectic cohomology (just as a group not a $R[u]$-module) is much easier to compute for general smooth divisor complement.
\end{remark}

\bibliographystyle{plain} 
\bibliography{ref}

\end{document}